\documentclass[12pt]{amsart}
\usepackage{amsfonts, amsthm, amssymb, amsgen, amsbsy, amstext, amsopn, amsmath,
latexsym, indentfirst, yfonts}

\renewcommand{\H}{\mathbb{H}}
\newcommand{\C}{\mathbb{C}}

\newcommand{\G}{\mathbb{G}}
\newcommand{\bbG}{\mathbb G}

\newcommand{\N}{\mathbb{N}}

\newcommand{\R}{\mathbb{R}}

\newcommand{\cG}{\mathcal{G}}

\newcommand{\lan}{\langle}
\newcommand{\ran}{\rangle}

\newcommand{\lra}{\longrightarrow}

\newcommand{\der}{\partial}

\newtheorem{teo}{Theorem}[section]
\newtheorem{lem}{Lemma}[section]
\newtheorem{rem}{Remark}[section]
\newtheorem{defin}{Definition}
\newtheorem{prop}{Proposition}[section]
\newtheorem{Cor}{Corollary}[section]

\begin{document}

\title
[Characterizations of differentiability for h-convex functions]
{{\bf Characterizations of differentiability for h-convex functions in stratified groups }}
\author{Valentino Magnani}
\address{Valentino Magnani, Dipartimento di Matematica \\
Largo Bruno Pontecorvo 5 \\ I-56127, Pisa}
\email{magnani@dm.unipi.it}
\author{Matteo Scienza}
\address{Matteo Scienza, Dipartimento di Matematica \\
Largo Bruno Pontecorvo 5 \\ I-56127, Pisa}
\email{scienza@mail.dm.unipi.it}

\begin{abstract}
Using the notion of h-subdifferential, we characterize both first and second
order differentiability of h-convex functions in stratified groups.
We show that Aleksandrov's second order differentiability of h-convex functions
is equivalent to a suitable differentiability of their horizontal gradient.
\end{abstract}
\maketitle

\tableofcontents

\footnoterule{
The first author has been supported by ''ERC ADG Grant GeMeTneES'' 
}

\pagebreak

\section{Introduction}

Convexity in sub-Riemannian Geometry is a quite recent stream, that goes back 
to the works by Danielli, Garofalo and Nhieu \cite{DGN2} and by Lu, Manfredi and Stroffolini \cite{LMS}.
All details and precise definitions related to convexity in stratified groups will be deferred to Section~\ref{SectionBN}.

Different pointwise notions of convexity have been investigated in \cite{DGN2}. 
Among them, the most natural turned out to be that of weakly h-convex function, in short, 
{\em h-convex function}.
An h-convex function $u:\Omega\lra\R$ defined on an open set $\Omega$ of a stratified group $\G$
satisfies the property of being classically convex, when restricted to all horizontal lines contained in $\Omega$. These are exactly the integral curves of the horizontal vector fields of $\G$.
We wish to stress that this notion of convexity turns out to be ``local'' and it does 
not require any assumption on $\Omega$.
In fact, it is not difficult to observe that smooth h-convex functions are characterized
by an everywhere nonnegative {\em horizontal Hessian}.
This fits with the approach of \cite{LMS}, where the authors introduce v-convex functions 
as upper semicontinuous functions, whose horizontal Hessian is nonnegative in the viscosity sense.
Let us point out that the notions of v-convexity and of h-convexity are equivalent, 
\cite{BalRic03}, \cite{Wang05}, \cite{JLMS}, \cite{Mag}.

There are various challenging questions on h-convex functions in stratified groups, 
that are still far from being understood. 
One of the most important is certainly the validity of an Aleksandrov-Bakelman-Pucci estimate, 
that is still an intriguing open question already in the Heisenberg group and it was also one of the main
motivations to study h-convexity in this framework, see \cite{DGN03} and \cite{DGN2}.

On another side, we have the second order differentiability of convex functions,
namely, the classical Aleksandrov-Busemann-Feller's theorem.
This is an important result in different areas of Analysis and Geometry. 
For instance, in the theory of fully nonlinear elliptic equations, this theorem plays an essential role in uniqueness theory, see
Chapter 5 of \cite{CafCab95}.

Since the works of Busemann and Feller, \cite{BusFel35}, and of Aleksandrov \cite{Alek39},
there have been different methods to establish this theorem in Euclidean spaces.
The functional analytic method by Reshetnyak, \cite{Resh68}, relies on the fact that
the gradient of a convex function has bounded variation.
This scheme can be extended to stratified groups, provided that one can prove
that an h-convex function is H-$BV^2$ in the sense of \cite{AmbMag}.
This important fact has been established by different authors for h-convex functions on
Heisenberg groups and two step stratified groups \cite{GutMon1}, 
\cite{GutMon2}, \cite{GarTou05}, 
\cite{DGNT} and also for $k$-convex functions with respect to 
two step H\"ormander vector fields, \cite{Trud05}.

Precisely, the main result of \cite{DGNT} gives us the following version of the
Aleksandrov-Busemann-Feller theorem. {\em Let $\Omega$ be an open set of a two step stratified
group and let $u:\Omega\lra\R$ be h-convex. Then $u$ has at a.e. $x\in \Omega$ 
a second order h-expansion at $x$}.
We mean that $u: \Omega \lra \R$  has a {\em second order h-expansion} at 
$x \in \Omega$ if there exists a polynomial $P_x:\G\lra\R$, whose homogeneous degree 
is less than or equal to two and such that 
\begin{equation}\label{alex}
u(xw) = P_x(w) + o(\|w\|^2).
\end{equation}
Unfortunately, it is still not clear whether h-convex functions are H-$BV^2$ in higher step groups
and this makes the Aleksandrov-Busemann-Feller's theorem an important open issue for the higher
step geometries of stratified groups.
On the other hand, the first proofs of this result in Euclidean spaces, \cite{Alek39}, \cite{BusFel35} 
and also some of the subsequent proofs did not use the bounded variation property of the gradient. 
For instance, the Rockafellar's proof of \cite{Rock85} relies on Mignot's a.e. differentiability of monotone functions, \cite{Mig76}, where the crucial observation is that the subdifferential of a convex function is a monotone function.

This may suggest different approaches to Aleksandrov's theorem in stratified groups
and constitutes our first motivation to study the properties of the h-subdifferential.
The notion of h-subdifferential has been introduced in \cite{DGN2} for h-convex functions.
In analogy with the local notion of convexity mentioned above, we use ``a local version'' of 
this notion, that allows us to treat h-convex functions on arbitrary open sets.

We say that $p\in H_1$ is an {\em h-subdifferential} of $u:\Omega\lra\R$ at $x\in\Omega$ if
whenever $h\in H_1$ and $[0,h]\subset x^{-1}\cdot\Omega$, 
where $[0,h]=\{th: 0\leq t\leq 1\}$, we have
\begin{equation}\label{hsub}
u(xh) \geq  u(x) + \left\langle p, h \right\rangle.
\end{equation}
We denote by $\partial_H u(x)$ the set of all h-subdifferentials of $u$ at $x$
and the corresponding set-valued mapping by $\der_Hu:\Omega\rightrightarrows H_1$.
According to notation and terminology of Section~\ref{SectionBN}, we represent a stratified group
$\G$ as a finite dimensional Hilbert space that is a direct sum of orthogonal subspaces 
$H_1$, $H_2$, $\ldots$, $H_\iota$ and that it is equipped with a suitable polynomial operation. Here $H_1$ is the subspace of {\em horizontal directions} at the origin and 
$\lan\cdot,\cdot\ran$ in \eqref{hsub} is the scalar product of $\G$.

A nice description of the various proofs present in the literature for
the Euclidean Aleksandrov-Busemann-Feller's theorem, along with a new proof, can be found 
in the paper by Bianchi, Colesanti and Pucci, \cite{BCP96}.
Here an interesting historical comment remarks that although the almost everywhere second order Taylor expansion
is proved in Aleksandrov's paper \cite{Alek39}, this fact is not stated as a theorem, whereas the 
almost everywhere differentiability of the gradient is more emphasized.

As our second motivation, we wish to clarify this point in general stratified groups. 
In the Euclidean framework, this has been done by Rockafellar,
where in Theorem~2.8 of \cite{Rock00} proves that {\em a convex function has a second order expansion at a 
fixed point if its gradient is differentiable at that point in the extended sense.}

We translate this notion in stratified groups saying that 
a locally Lipschitz function $u : \Omega \rightarrow \R$ is 
{\em twice h-differentiable at $x$} if it is h-differentiable at $x$ and there exists 
an h-linear mapping $A_x : \G\rightarrow H_1$ such that 
\begin{equation}\label{difgrad}
\left\|\frac{\nabla_H u(xw)-\nabla_H u(x)-A_x(w)}{\|w\|}\right\|_{L^\infty(B_\delta,H_1)}\lra0
\quad\mbox{as}\quad \delta\to0^+\,.
\end{equation}
We also say that $\nabla_Hu$ is {\em h-differentiable at $x$ in the extended sense}.
This notion makes sense, since Lipschitz functions are almost everywhere
h-differentiable, by Pansu's result \cite{Pan89}.
We are now in the position to state the main result of this paper.
\begin{teo}[Second order characterization]\label{equivalence}
Let $u:\Omega\lra\R$ be h-convex and let $x\in\Omega$. Then $u$ has a second order
h-expansion at $x$ if and only if it is twice h-differentiable at $x$.
In addition, in this case the following facts hold
\begin{enumerate}
\item
the gradient $\nabla_{V_2}u(x)=\big(X_{m_1+1}u(x),\ldots,X_{m_2}u(x)\big)$ 
of $u$ at $x$ along $V_2$ exists, where 
$(X_{m_1+1},\ldots,X_{m_2})$ is an orthonormal basis of the second layer $V_2$,
\item
denoting by $P_x$ the second order h-expansion of $u$ at $x$, we have
\[
P_x(w)=u(x)+\big\lan\big(\nabla_Hu(x)+\nabla_{V_2}u(x)\big),w\big\ran+
\frac{1}{2}\,\lan\nabla_H^2P_x\,w,w\ran
\]
\item
denoting by $A_x$ the h-differential of $\nabla_Hu$ in the extended sense at $x$,
then its connection with $P_x$ is given by the formula
\[
\big(\nabla_H^2P_x\big)_{ij}=(A_x)^i_j-\sum_{l=m_1+1}^{m_2}  a^{li}_j\;X_lu(x)\,,
\]
where $a^{li}_j$ only depend on the coordinates of the group and appear in 
\eqref{X_j2}, the horizontal Hessian $\nabla_H^2P_x$ is nonnegative and $X_iX_jP_x=(A_x)^i_j$.
\end{enumerate}
\end{teo}
As a consequence of this theorem, we can establish that 
{\em the horizontal gradient of h-convex 
functions in two step stratified groups are almost everywhere h-differentiable in the extended sense
and satisfy the properties (1), (2) and (3) of Theorem~\ref{equivalence}}. 

We also wish to point out how the formula of (3) in commutative groups fits into Rockafellar's
result on symmetry and nonnegativity of $A_x$, \cite{Rock00}. This symmetry breaks in
stratified groups, although the symmetric part $\nabla_H^2u$ of the h-differential 
in the extended sense remains nonnegative for any h-convex function $u$.
This is due to the fact that the extended differential $A_x$ also takes into account
the first order derivatives along second order directions, as it happens for $P_x$.

The proof of Theorem~\ref{equivalence} needs several basic results involving the h-subdifferential. 
Since we expect that these results should play a role in the potential
development of a nonsmooth calculus for h-convex functions, we wish to emphasize some of them. 
We follow Rockafellar's approach to show that the existence of a second order h-expansion
implies the h-differentiability of the horizontal gradient in the extended sense.
To this aim, we have first to establish the following
\begin{lem}\label{MIeqGR}
Let $u:\Omega\lra\R$ be h-convex. 
Then $u$ is twice h-differentiable at $x$ if and only if there exist an h-linear mapping  
$A_x : \G\rightarrow H_1$ and $v \in H_1$ such that 
\begin{equation}\label{mignot}
\partial_H u(xw) \subseteq v + A_x(w) + o(\|w\|)\mathbb B
\end{equation}
for all $w\in x^{-1}\Omega$. In particular, if \eqref{mignot} holds, then $v = \nabla_H u(x)$.
\end{lem}
At first sight, extended differentiability in the sense of \eqref{mignot} seems stronger than
\eqref{difgrad}, that implies a convergence up to a negligible set, where $\nabla_Hu$ is not
defined. In fact, the delicate point is to prove that extended differentiability implies 
\eqref{mignot}. This is a consequence of the following characterization of the
h-subdifferential.
\begin{teo}\label{carsubdif}
Let $u : \Omega\rightarrow \R$ be h-convex. Then for every $x \in \Omega$ we have
\begin{equation}\label{envstar}
\bar{co}\left( \nabla_H^\star u(x)\right) = \partial_H u (x)\,.
\end{equation} 
\end{teo}
We denote by $co(E)\subset H_1$ the convex hull in $H_1$ of the subset $E\subset H_1$
and by $\bar{co}(E)$ its closure. The {\em h-reachable gradient} is given by
\begin{equation}\label{limgrad}
\nabla_H^\star u(x) = \Big\lbrace  p\in H_1: \, x_k\rightarrow x,\, 
\nabla_H u(x_k)\; \mbox{exists for all $k$'s and}\; \nabla_Hu(x_k)\to p\; \Big\rbrace \,.
\end{equation}
The proof of equality \eqref{envstar} in the Euclidean case can be found for instance in \cite{AmbDan00}.
There are two main features in the proof of Theorem~\ref{carsubdif}, with respect to the Euclidean one.
First, it is still possible to use the Hahn-Banach's theorem, when applied inside the horizontal subspace $H_1$, that has a linear structure. 
Second, the group mollification does not commute with horizontal derivatives, hence 
the mollification argument of the Euclidean proof cannot be applied. We overcome this point
by a Fubini type argument with respect to a semidirect factorization, 
following the approach of \cite{Mag09}. 
The h-differentiability of $u$ from validity of \eqref{mignot} is a consequence of the following 
\begin{teo}[First order characterization]\label{UniqueSub}
Let $u:\Omega\lra\R$ be h-convex. Then $u$ is h-differentiable at $x$ if and only
if $\partial_Hu(x)=\{p\}$ and in this case $\nabla_H u(x)= \left\lbrace p \right\rbrace$.
\end{teo}
The uniqueness of the h-subdifferential as a consequence of h-differentiability has been already shown
\cite{DGN2}, see also \cite{CaPi1} for the case of Heisenberg groups.
To show the opposite implication we decompose the difference quotient of $u$ 
into sums of difference quotients along horizontal directions. 
The same decomposition along horizontal directions have been first used by Pansu, \cite{Pan89}.
The second ingredient is the following 
\begin{teo}[Nonsmooth mean value theorem]\label{meanV}
Let $u:\Omega\lra\R$ be an h-convex function. Then for every $x \in \Omega$ and every $h$ such that 
$[0,h]\subseteq H_1\cap x^{-1}\Omega$,  there exists $t \in [0,1]$ and 
$p \in \partial_H u (x\delta_th)$ such that $u(xh) - u(x) =  \left\langle p ,h \right\rangle$.
\end{teo}
This theorem is also important to complete the characterization of Theorem~\ref{equivalence}.
In fact, it is an essential tool to establish that twice h-differentiability implies 
the existence of a second order $h$-expansion. 
This implication again requires Pansu's approach to differentiability
and in addition a nonsmooth mean value theorem
for functions of the form $U+P$, where $U$ is h-convex and $P$ is a polynomial of 
homogeneous degree at most two. This slightly more general version of Theorem~\ref{meanV}
is given in Theorem~\ref{meanval}, where the h-subdifferential is replaced
by the more general $\lambda$-subdifferential, see Definition~\ref{lambdasub}.
In the Euclidean framework, a short proof of the previous result can be found in Theorem~7.10
of \cite{AlbAmb99}, where the Clarke's nonsmooth mean value theorem plays a key role.

In this connection, we wish to emphasize the intriguing open question on the 
validity of a nonsmooth mean value theorem for Lipschitz functions in stratified groups.
In the Euclidean framework, this theorem holds using the notion of Clarke's differential.
This notion of differential relies on subadditivity of ``limsup directional derivatives'',
that allows in turn to apply Hahn-Banach's theorem, see \cite{Clarke90}. 
The obvious extension of this notion to stratified groups does not work and
the analogous obstacle comes up considering h-convex functions, where horizontal
directional derivatives always exist, see Definition~\ref{par+}.
It is curious to notice that our nonsmooth mean value theorem implies this 
subadditivity, see Corollary~\ref{subadd}, whereas in the Euclidean framework 
subadditivity eventually leads to the nonsmooth mean value theorem.

\medskip

{\bf Acknowledgments.}
We are grateful to Andrea Calogero and Rita Pini for having addressed our attention to
the paper by Rockafellar \cite{Rock00}, that was our starting point.
We thank Luigi Ambrosio for a stimulating conversation and for having pointed out 
to us the notion of $\lambda$-subdifferential
in connection with the characterization of second order differentiability. 
\section{Basic notions}\label{SectionBN}
A stratified group can be thought of as a graded vector space $\G=H_1\oplus\cdots\oplus H_\iota$ with a polynomial group operation given by the Baker-Campbell-Hausdorff formula. More precisely, let $\cG$ be its Lie algebra, $n= \rm{dim}\cG = m_\iota$. Then we assume that $\cG = V_1\oplus\ldots\oplus V_\iota$, where $V_j=[V_1,V_{j-1}]$ for all $j\geq1$ and $V_j=\{0\}$ if and only if $j>\iota$. On $\G$ we can define a natural family of dilation $\delta_r : \G \rightarrow \G$ compatible with the group operation, \cite{FS82}.\\
The left invariant vector fields of $V_j$ are exactly the ones that at the origin take values
in $H_j$. Recall that the origin is exactly the unit element of the group. A scalar product on $\G$ will be understood, assuming that all subspaces $H_j$ are orthogonal. We denote by $\pi_j:\G\lra H_j$ the associated orthogonal projections. 
For every $s= 1,\ldots \iota$, we fix a basis $(e_{m_{s-1} +1},\ldots, e_{m_s})$ of $H_s$, then 
\[
\sum_{i = m_{s-1} +1}^{m_s} \!x_i e_i\, \in H_s\quad\mbox{and}\quad
x=\sum_{s=1}^\iota \sum_{i = m_{s-1} +1}^{m_s} \!x_i\, e_i. 
\]
We also fix $\left(X_{m_{s-1}+1},\ldots,X_{m_s}\right)$ as the basis of $V_s$ such that, with respect to the coordinates $(x_j)$, $X_j$ is  $e_j$.
Throughout, we fix an homogeneous distance $d$ on $\bbG$, i.e. a continuous map $d:\bbG \times \bbG \rightarrow [0,+\infty[$ that makes $(\bbG,d)$ a metric space and has the following properties
\begin{enumerate}
\item
$d(x,y)=d(ux,uy)$ for every  $x,y,u \in \G$, 
\item
$d(\delta_r x, \delta_r y)=r d(x,y)$ for every $r>0.$
\end{enumerate}
For every $w\in \bbG$, we denote by $\| w \|$ the homogeneous norm of $w$ induced by the distance $d$
by $\| w \| = d(0,w)$.

\vskip.2cm

As in \cite{FS82}, open balls with respect to $d$ will be denoted by $B_{x,r}$.
The following proposition is a well known fact, see for instance \cite{MagTh}.
\begin{prop}\label{genprop}
Let $\bbG$ be a stratified group and let $(w_1,\ldots,w_{m_1})$ be a basis of $H_1$. Then there exists a positive integer $\gamma$ and an open bounded neighbourhood of the origin $U \subset\R^\gamma$ such that the following set
$
W =\big\{\prod_{s=1}^\gamma a_s w_{i_s} \vert \; (a_s) \subset U \big\},
$
where $1\leq i_s\leq m_1$ for all $s=1,\ldots,\gamma$, is an open neighbourhood of $0 \in \bbG$.
\end{prop}
According to notation of the previous proposition, we set the geometric constant
\begin{equation}\label{M}
M = \sup_{y \in W} \|y\|.
\end{equation}
\begin{defin}[h-convex set]\rm
We say that a subset $C\subset\G$ is h-convex if for every $x,y \in C$ such that 
$x \in  H_y$ we have $x\delta_\lambda(x^{-1}y) \in C$ for all $\lambda \in [0,1]$. 
\end{defin}
We denote by $H_x$ the left translation of $H_1$ by $x$, namely $H_x=xH_1$.
For each $h \in H_1$, we define the {\em horizontal segment} 
$ \left\lbrace t h ,\ t \in [0,1]\right\rbrace$ through the short notation $[0,h]$.
For any $x \in \G$, we set $x\cdot [0,h] = \lbrace x\delta_t h,\ 0\leq t \leq 1\rbrace$
and throughout $\Omega$ denotes an open subset of $\G$. 
\begin{defin}[h-convex function]\label{HCfunction}\rm
We say that $u: \Omega \rightarrow \R$ is {\em h-convex} if for every $x,y\in\Omega$ such that
$x \in H_y$ and $x\cdot[0,x^{-1}y]\subset\Omega$, we have
\begin{equation}
u\big(x\delta_\lambda(x^{-1}y)\big) \leq \lambda u(y) + (1-\lambda) u(x), \qquad \forall \lambda \in [0,1].
\end{equation}
\end{defin}
As an important property of h-convex functions, we have the following
\begin{teo}[M. Rickly, \cite{Rickly06}]\label{rickly}
Every measurable h-convex function is locally Lipschitz.
\end{teo}
Throughout, all h-convex functions are assumed to be measurable, hence locally Lipschitz. 
\begin{defin}\rm 
We say that $u:\Omega\lra\R$ is {\em h-differentiable} at $x\in\Omega$,
if there exists an h-linear mapping $L:\G\lra\R$, namely, a linear map such that
$L(x) = L(\pi_1(x))$, such that 
$ u(xz)= \nobreak u(x)+L(z)+o(\|z\|)$. Notice that $L$ 
is unique and its associated vector with respect to the scalar product
is denoted by $\nabla_Hu(x)$. 
\end{defin}
%
%
%
%
%
%

\begin{defin}\rm
We say that $ P : \bbG \rightarrow \R$ is a \textit{polynomial} on $\bbG$, if with respect to some fixed
graded coordinates we have $P(x) = \sum_{\alpha \in \mathcal A} c_\alpha x^\alpha$, under the convention
$x^\alpha = \prod_{i=1}^{n} x^{\alpha_i}_i,$ and $x_j^0=1$, where $\mathcal A\subset\N^n$ is a finite set. 
The {\em homogeneous degree} of $P$ is the integer $h$-$\mbox{deg} (P) = \max \left\lbrace d(\alpha),\ \alpha \in \mathcal A \right\rbrace$, where $d(\alpha) = \sum d_i \alpha_i$, and $d_i = s$ if 
$m_{s-1}+1 \leq i \leq m_s$.
\end{defin}
By the previous definitions, any polynomial $P$ can be decomposed into the sum of its
{\em $j$-homogeneous parts}, denoted by $P^{(j)}$, hence 
\[
P=\sum_{0\leq j \leq h\mbox{-\scriptsize deg}P} P^{(j)}.
\]
A polynomial is {\em j-homogeneous} if it coincides with its $j$-homogeneous part.
\begin{defin}[$C^k_H-$maps, \cite{MagTh}] \rm We say that $f : \Omega \rightarrow \R$ is h-continuously differentiable in $\Omega$ if it is differentiable at $x \in \Omega$ and $d_Hf : \Omega \rightarrow HL(\G,\R)$ is continuous, where $HL(\G,\R)$ is the space of h-linear map. We denote by $C^1_H(\Omega)$ the space of all continuously differentiable maps. By induction on $k\geq 2$ we say that $f:\Omega \rightarrow \R$ is h-continuously {\em k-differentiable} if the $(k-1)$ h-differential $d^{k-1}_Hf : \Omega \rightarrow HL(\G,HL^{k-2}(\G,\R))$ is h-continuously differentiable. We denote by $C^k_H(\Omega)$ the space of all continuously $k$-differentiable maps.
\end{defin}

\begin{teo}[Stratified Taylor Inequality, Theorem 1.42 in \cite{FS82}]\label{taylor}
For each positive integer $k$ there is a constant $C_k$ such that for all $f \in C^k_H(\Omega)$ and all $x,y \in \Omega$,
\begin{equation}\nonumber
\left | f(xy) - P_x(y)\right | \leq C_k \|y\|^k \eta(x,b^k\|y\|),
\end{equation}
where $P_x$ is the left Taylor polynomial of $f$ at $x$ of homogeneous degree $k$, $b$ is a constant depending only on $\bbG$, and for $r>0$,
$$
\eta(x,r) = \sup_{\|z\|\leq r, d(I) = k} \left |X^If(xz)-X^If(x) \right |,
$$
where $X^I = X_{i_1}\cdots X_{i_l}$, for a certain $l$ dependent on $I$ and $(i_1,\ldots,i_l) \in \{1,\ldots,m_1\}^{l}$.
\end{teo}

As in \cite{FS82}, given $a \in \N$, we shall denote by $\mathcal P_a$ the space of polynomials of homogeneous degree $\leq a$. Moreover, by Proposition 1.25 in \cite{FS82}, $\mathcal P_a$ is invariant under left translations. 
\begin{prop}[1.30 in \cite{FS82}]\label{isom}
Suppose $a\in \N$, and let $\mu = \mbox{dim} \mathcal P_a$. Then the map
$$
P \rightarrow (X^IP(0))_{d(I)\leq a},
$$
is a linear isomorphism from $\mathcal P_a$ to $\C^\mu$.
\end{prop}
In particular, we are interested to find the explicit isomorphism of the previous proposition in the
case of real polynomials of homogeneous degree less than or equal to two. Let $P$ be of a
2-homogenous polynomial with expression
\[
P(x)=\frac{1}{2}\sum_{1\leq i,j\leq m_1}c_{ij}\,x_ix_j+\sum_{s=m_1+1}^{m_2}c_s\,x_s
\]
and let us consider, with respect to the same system of graded coordinates, 
the left invariant vector fields
\[
X_j=\der_{x_j}+\sum_{l=m_{d_j}+1}^na^l_j(x)\der_{x_l}
\]
for $j=1,\ldots,n$, where $a^l_j(x)$ are $(d_l-d_j)$-homogeneous polynomial.
Then a direct computation gives us the following formula
\begin{equation}\label{exppol}
P(x)=\lan\nabla_{V_2}P,x\ran+\frac{1}{2}\lan\nabla_H^2P x,x\ran\,,
\end{equation}
where $\nabla_{V_2}P=(X_{m_1+1}P,\ldots,X_{m_2}P)$ is constant since it is 0-homogeneous, we set $\lan\nabla_{V_2}P,x\ran=\sum_{j=m_1+1}^{m_2}X_jP\,x_j$ and furthermore
\begin{equation}
(\nabla_H^2P)_{ij}=\frac{X_iX_jP+X_jX_iP}{2}
\end{equation}
denotes the coefficient of the so-called {\em symmetrized horizontal Hessian},
that is also 0-homogeneous, hence constant.
In fact, the explicit expression of $X_j$ immediately yields $X_jP=c_j$ for all
$j=m_1+1,\ldots,m_2$. To check equality
\begin{equation}\label{c_ij}
\frac{c_{ij}+c_{j\,i}}{2}=\frac{X_iX_jP+X_jX_iP}{2}
\end{equation}
for $1\leq i,j\leq m_1$, we observe that 
\begin{equation}\label{X_j2}
X_j(x)=\der_{x_j}+\sum_{l=m_1+1}^{m_2}\sum_{i=1}^{m_1}a^{li}_j\,x_i\,\der_{x_l}+
\sum_{l=m_2+1}^n a^l_j(x)\,\der_{x_l}  
\end{equation}
since $a^l_j(x)=\sum_{i=1}^{m_1}a^{li}_j\,x_i$ is 1-homogeneous for
$d_l=2$ and $d_j=1$.
Taking into account the previous expression, we arrive at the following
\[
X_jP(x)=\frac{1}{2}\sum_{i=1}^{m_1} (c_{ij}+c_{j\,i})\,x_i+\sum_{i=1}^{m_1}
\sum_{l=m_1+1}^{m_2} X_lP\,a^{li}_j x_i
\]
that immediately yields
\begin{equation}\label{alij}
X_iX_jP=\frac{c_{ij}+c_{j\,i}}{2}+\sum_{l=m_1+1}^{m_2} X_lP\;a^{li}_j\,.
\end{equation}
Finally, formula \eqref{c_ij} follows by the equality $a^{li}_j=-a^{lj}_i$.
This is in turn a consequence of the Baker-Campbell-Hausdorff formula
for the second order bilinear terms.

\begin{rem}\rm\label{uniqueP}
Let $P$ be a polynomial of homogeneous degree at most 2, and suppose that $P(0)= p_0$ and $X_iP(x)= l_i(x)$, for $i = 1,\ldots,m_1$ where $l_i :\bbG \rightarrow \R$ are h-linear functions. Clearly we can compute $(X^\alpha P)(0)$ for each multiindex $\alpha$, $d(\alpha) \leq 2$, then by the previous proposition $P$ is uniquely determined.
\end{rem}
%
%

%
\vskip 0,3 cm
\begin{rem}\label{poly}\rm
Let $P: \bbG \rightarrow \R$ be a polynomial of homogeneous degree at most 2.
Let $P^{(2)}(x)$ the 2-homogeneous part of $P$, and define 
\[
\displaystyle \lambda = \max_{\|w\| = 1}\vert P^{(2)}(w) \vert.
\] 
For every $ 1\leq i,j \leq m_1$, we have the constants $X_iX_j P  = c_{i,j}$
and $X_iX_j(P(xh)) = c_{i,j}$ for every $x,h \in \G$. This is a consequence of the following general fact, given a smooth function $u$ and $X$, a left invariant vector fields on $\G$, then $ X(u(xh)) = (Xu)(xh)$. 
Consider $P(xh)$ as a function of $h$, applying Theorem \ref{taylor} we get a polynomial $P_x(h)$ such that
$$
P(xh) = P_x(h) + o(\|h\|^2).
$$
Notice that by the left translation invariance of $\mathcal P_2$, $P(xh)$ as a function of $h$ is a polynomial of homogeneous degree at most 2, hence $P(xh) = P_x(h)$. Clearly $P^{(0)}_x (h) = P(x)$ and $P^{(1)}_x(h) = \left\langle \nabla_H P(x),h\right\rangle$, as a consequence
\begin{equation}\label{p1}
P(xh)- P(x)- \left\langle \nabla_H P(x),h\right\rangle = P^{(2)}_x(h).
\end{equation}
By \eqref{p1} and previous considerations it follows that 
$$
c_{i,j} = X_iX_jP(h) = X_iX_jP^{(2)}(xh) = X_iX_jP^{(2)}_x(h),\quad i,j = 1,\ldots, m_1.
$$
Moreover all the other derivatives of $P_x^{(2)}$ are zero, thus we can conclude that $P^{(2)}_x(h) = P^{(2)}(h)$ by Proposition \ref{isom}. Finally we get
\begin{eqnarray}\nonumber
P(xh) &=& P(x) + \left\langle \nabla_HP(x),h \right\rangle + P^{(2)}(h)\\ \nonumber
&\geq& P(x) + \left\langle \nabla_HP(x),h \right\rangle - \lambda \|h\|^2.
\end{eqnarray}
\end{rem}
%
%
%
%
%

\section{Properties of the h-subdifferential}
In the sequel, $\mathbb B$ will denote the unit ball in $H_1$
centered at the origin with respect to the fixed scalar product on $\G$.
\begin{rem} \rm The set $\partial_Hu(x) \subset H_1$ is convex, in fact let $p,q \in \partial_H u(x)$ and choose $\lambda \in [0,1]$. Then $\lambda p + (1-\lambda)q \in \partial_H u(x)$, this follows adding the two inequalities 
\begin{eqnarray}\nonumber
\lambda u(xh) &\geq&  \lambda u(x) + \left\langle \lambda p, h \right\rangle \\ \nonumber
(1-\lambda)u(xh) &\geq&  (1-\lambda) u(x) + \left\langle (1-\lambda) q, h \right\rangle.
\end{eqnarray}
Moreover, let $u$ be an h-convex function in $\Omega$, then by Theorem \ref{rickly} $u$ is locally Lipschitz. Hence, for every $B_{x,r} \subseteq \Omega$, there exists $L>0$ depending on $x$ and $r>0$
such that
\begin{equation}\label{equibound}
\partial_Hu(y) \subseteq L \mathbb B \quad \mbox{for every } y \in B_{x,r}.
\end{equation}
\end{rem}
\begin{rem}\label{derudif}\rm 
At any h-differentiability point $x$,
there holds $\der_Hu(x)=\{\nabla_Hu(x)\}$, as noticed in \cite{DGN2}.
\end{rem}
Throughout, we use the symbol $co$ to denote the linear convex envelope in $H_1$.
Then our first important tool is the following
\begin{teo}\label{1carsubdif}
Let $u : \Omega\rightarrow \R$ be h-convex. Then for every $x \in \Omega$ we have
\begin{equation}\label{1envstar}
\partial_H u (x)\subseteq\bar{co}\left( \nabla_H^\star u(x)\right)\,.
\end{equation} 
\end{teo}
\begin{proof}
Suppose that there exists $p \in \partial_H u(x)$ such that 
$ p \notin \bar{co}\left( \nabla_H^\star u(x)\right)$.
We can assume that $p = 0$, otherwise one considers $v(x) = u(x) - \left\langle p, \pi_1 (x)\right\rangle$, that is still h-convex.
Since $\bar{co}\left( \nabla_H^\star u(x)\right)$ is a closed convex subset of $H_1$,
the Hahn-Banach separation theorem can be applied to this set and the origin,
hence there exists $q \in H_1$, $d(0,q)=1$, and $\alpha >0$ such that
\begin{equation}\label{hbappl}
\left\langle z, q\right\rangle > \alpha \qquad \forall z \in  \nabla_H^\star u(x).
\end{equation}
We claim the existence of $r >0$ such that  $B_{x,r} \subset\Omega$ and $ \left\langle \nabla_H u(y), q\right\rangle > \frac{\alpha}{2}$ for every $y \in B_{x,r}$ where $u$ is h-differentiable. By contradiction, suppose there exist sequences $r_j \rightarrow 0$ and $y_j \in B_{x,r_j}$ such that $ \left\langle \nabla_H u(y_j), q\right\rangle \leq \frac{\alpha}{2}$, then possibly passing to a subsequence we have $y_j \rightarrow x$ and $\nabla_H u(y_j) \rightarrow z \in \nabla_H^\star u(x)$, with $\left\langle z, q\right\rangle \leq \frac{\alpha}{2}$ and this conflicts
with \eqref{hbappl}. Denote by $r$ the positive number having the previous property. 
Let $Q = \left\lbrace \delta_t q \ :\ t \in \R \right\rbrace$ and consider $\mu$ the Haar measure on $\bbG$. By Proposition 2.8 in \cite{Mag09} there exists a normal subgroup $ N \subset \bbG$, such that $N\rtimes Q = \bbG$. Moreover there exist  $\nu_q$ and $\mu_N$, respectively Haar measures on $Q$ and $N$ such that for every measurable set $A \subset \bbG$
\begin{equation}\label{fubG}
\mu(A) = \int_N \nu_q(A_n)\ d\mu_N(n)
\end{equation}
where $A_n = \left\lbrace h \in Q \ :\ nh \in A\right\rbrace$. Let $P$ be the set of h-differentiable points of $u$, which has full measure in $\Omega$. From \eqref{fubG} it follows that for $\mu_N$-a.e. $n \in N$, $\nu_Q(Q\setminus n^{-1}P) = 0$. Then for $\mu_N$-a.e. $n\in N$, $n\delta_tq \in P$ for a.e. $t \in \R$. Let $\bar n \in N $ and $\delta_{\bar t}q \in Q$ respectively the unique elements in $N$ and $Q$ such that $x = \bar n \delta_{\bar t}q$. Let $\epsilon >0 $ and $s >0$ such that $B^N_{\bar n,s}\cdot B^Q_{\delta_{\bar t}q,\epsilon} \subset B_{x,r}$, where $B^N_{\bar n,s}$ and $B^Q_{\delta_{\bar t}q,\epsilon}$ are open balls respectively in $N$ and $Q$. Fix a point $n \in B^N_{\bar n,s}$ where $u(nh)$ is $\nu_q$-a.e. differentiable and consider the convex function $v(t) = u(n\delta_tq)$, for $\nu_q$-a.e. $\delta_tq$, $t \in (-\epsilon + \bar t ,\epsilon + \bar t)$ we have
\begin{equation}\nonumber
v'(t) = \left\langle \nabla_H u (n\delta_tq), q \right\rangle > \frac{\alpha}{2}.
\end{equation}
Integrating the previous inequality, taking into account the Lipschitz regularity of $v$ we get
\begin{equation}\nonumber
v(t_1)-v(t_2) = u(n\delta_{t_1}q)-u(n\delta_{t_2}q) > \frac{\alpha}{2}(t_1-t_2)
\end{equation}
where $-\epsilon + \bar t < t_2 < t_1< \epsilon + \bar t$. Now let $n_j \rightarrow \bar n \in B^N_{\bar n,s}$ such that $n_jh$ is a differentiable point of the map $h \rightarrow u(n_jh)$ for every $j$ and $\nu_q$-a.e. $h$, by the previous considerations we have
\begin{equation}\nonumber
u(n_j\delta_{t_1}q)-u(n_j\delta_{t_2}q) > \frac{\alpha}{2}(t_1-t_2) \qquad -\epsilon + \bar t< t_2 < t_1< \epsilon + \bar t
\end{equation}
finally we can pass to the limit in $j$ and get the strict monotonicity of $u(\bar n \delta_tq)$ i.e.
\begin{equation}\label{monotone}
u(\bar n \delta_{t_1}q)-u(\bar n \delta_{t_2}q) \geq \frac{\alpha}{2}(t_1-t_2) \qquad -\epsilon + \bar t< t_2 < t_1< \epsilon + \bar t.
\end{equation}
Recall that $0 \in \partial_H u(x)$, i.e. $u(xh) \geq u(x)$ whenever $[0,h] \subseteq H_1\cap x^{-1}\Omega$.
Thus, $u(\bar n \delta_tq) \geq u(\bar n \delta_{\bar t}q)$ for all 
$t \in (\bar t -\epsilon,\bar t + \epsilon)$, in contrast with the monotonicity \eqref{monotone}.
\end{proof}
Joining Theorem \ref{1carsubdif} with Theorem 9.2 of \cite{DGN2}, we immediately get
\begin{Cor}
Let $u : \Omega \rightarrow \R$ be an h-convex function. There exists $C = C(\bbG) >0 $ such that for every ball $B(x,r) \subset \bbG$  one has
\begin{equation}\label{GarSub}
\sup_{ \substack{ p \in \partial_H u(y)\\ y \in  B_{x,r} }} \vert p \vert  
\leq \dfrac{C}{r} \dfrac{1}{\vert B_{x,15r} \vert} \int_{B_{x,15r}} \vert u(y) \vert dy.
\end{equation}
\end{Cor}
Given a set $E \subset \G$ and $\rho >0$, by $I(E,\rho)$, we denote the open set
$$
I(E,\rho) = \left \lbrace x \in \G ,\ d(x,E) < \rho \right \rbrace.
$$
\begin{prop}\label{unifConv}
Let $u_i : \Omega \rightarrow \R$ be a sequence of h-convex functions, $\Omega \subset \bbG$ open. Suppose that $u_i$ uniformly converge on
compact sets to an h-convex function $u$. Let $x \in \Omega$ and let $(x_i)$ be a sequence
in $\Omega$ converging to $x$. Then for every $\epsilon >0$, there exists $i_0 \in \mathbb N$ such that
\begin{eqnarray}\label{unifSub}
&&\partial_H u_i(x_i) \subseteq \partial_H u(x) + \epsilon \mathbb B \quad \mbox{for all}\;\; i \geq i_0.
\end{eqnarray}
In addition, if $u$ is everywhere h-differentiable in $\Omega$, then for every compact set $K \subset \Omega$ and every $\epsilon >0$, there exist $i_{\epsilon,K}$ such that
\begin{eqnarray}\label{unifSub2}
&&\partial_H u_i(y) \subseteq \nabla_H u(y) + \epsilon \mathbb B \quad
\mbox{for all} \;\; i \geq i_{\epsilon,K}, \;\mbox{whenever}\; y\in K.
\end{eqnarray}
\end{prop}
\begin{proof}
We argue by contradiction in both cases,  hence we suppose that there exist $\epsilon >0$ and a subsequence $p_{i_k} \in \partial_H  u_{i_k} (x_{i_k})$ such that for every $p \in \partial_H u(x)$ we have $ \vert p_{i_k} - p\vert > \epsilon$. By estimate \eqref{GarSub} one easily observes that the sets $\partial_H u_i(x_i)$ are equibounded, thus possibly passing to a subsequence, $p_{i_k} \rightarrow \nolinebreak q$ and $\mbox{dist}(p_{j_k},\partial_Hu(x_{j_k}))\geq \epsilon$. Define a monotone family of compact sets  $K_\tau = \left\lbrace x \in D_{\tau }\ :\ d(x,\Omega^c) \geq \frac{1}{\tau} \right\rbrace$, such that $\bigcup_{\tau>0}K_\tau = \Omega$. Let $j_l$ be a subsequence such that $ p_{j_l} \rightarrow q$ and $ \Vert u_{i_l} - u \Vert_{L^\infty(K_{l})} < \frac{1}{l}$. Recall that $p_{j_l} \in \partial_H u_{j_l}(x_{j_l})$, then
\begin{equation}\nonumber
u_{j_l}(x_{j_l}h) \geq u_{j_l}(x_{j_l}) + \left\langle p_{j_l}, h\right\rangle \qquad \mbox{whenever } [0,h] \subseteq H_1\cap x_{j_l}^{-1}\Omega.
\end{equation}
By uniform convergence for $l$ sufficiently large, we get
\begin{equation}\label{unif1}
u(x_{i_l}h) \geq u(x_{i_l}) - \frac{2}{l} + \left\langle p_{i_l},h\right\rangle \qquad \mbox{whenever } [0,h] \subseteq H_1\cap x_{i_l}^{-1}K_{l}.
\end{equation}
Take $[0,h] \subseteq (x^{-1}\Omega)\cap H_1$, then there exists $l_0$ such that for every $ l > l_0$, $[0,h] \subset x^{-1}K_{l}\cap H_1$. Since $\Omega$ is an open set there exists $\rho > 0$ such that $ I(x\cdot[0,h],\rho) \subset K_{l} $. By continuity of left translation there exists $j(\rho)$ such that for every $j_l > j(\rho)$,
$$
x_{j_l}\cdot[0,h] \subseteq I(x\cdot[0,h],\rho),
$$
hence $[0,h] \subseteq x_{j_l}^{-1}K_{l}$. Then \eqref{unif1} holds with $h$ and passing to 
the limit in $l$ we get
\begin{equation}\label{in1}
u(xh) \geq u(x) + \left\langle q,h\right\rangle,
\end{equation}
thus $q \in \partial_H u(x)$, getting a contradiction. Now suppose that $u$ is everywhere h-differentiable. Again, by contradiction there exist a compact set $W \subset \Omega$, $\epsilon > 0$ and a subsequence $j_l$ such that for all $l$, $x_{j_l}\in W$ we have 
$$
\partial_H u_{j_l}(x_{j_l}) \nsubseteq \partial_Hu(x_{j_l}) + \epsilon\mathbb B.
$$
Then, we can find $p_{j_l} \in \partial_H u_{j_l}(x_{j_l})$ such that $\mbox{dist}(p_{j_l},\partial_Hu(x_{j_l})) \geq \epsilon$, for all $l>0$. As before, we can suppose that, possibly passing to a subsequence, $x_{j_l} \rightarrow \bar{x} \in W$ and $p_{j_l} \rightarrow \bar{p}$. By Remark 2 and h-differentiability at $\bar x$, taking into account the first part of this proposition, we get that for every $\gamma >0$ there exists $j_{l'}$ such that 
\begin{eqnarray}\nonumber
\partial_H u_{j_l}(x_{j_l}) &\subset& \nabla_Hu(\bar{x}) + \gamma \mathbb B\\ \nonumber
\partial_H u(x_{j_l}) &\subset& \nabla_H u(\bar{x}) + \gamma \mathbb B, \qquad \forall j_l > j_l'.
\end{eqnarray}
From the previous inclusions, it follows that
$$
\epsilon \leq \mbox{dist}(p_{j_k}, \partial_H u(x_{j_k})) \leq 2\gamma.
$$
If we choose $\gamma = \frac{\epsilon}{4}$, then reach a contradiction, concluding the proof.
\end{proof}
Taking the constant sequence $u_i = u$ in the previous proposition
and taking into account \eqref{unifSub}, we immediately reach the following
\begin{Cor}\label{closesub}
Let $\Omega$ be an open set of $\G$ and let $u:\Omega\rightarrow \R$ be an h-convex function, 
then $\partial_H u : \Omega \rightarrow \mathcal P (H_1)$ has closed graph.
\end{Cor}
The previous corollary allows us to complete the proof of Theorem~\ref{carsubdif}.
\begin{proof}[Proof of Theorem~\ref{carsubdif}]
By virtue of Theorem~\ref{1carsubdif}, we have only to prove the inclusion
\[
\bar{co}\left( \nabla_H^\star u(x)\right) \subseteq \partial_H u (x).
\] 
By Corollary~\ref{closesub}, the set-valued map $\partial_H u$ has closed graph
and $\partial_H u (y) = \left\lbrace \nabla_H u(y)\right\rbrace$ at any h-differentiability
point $y$ of $u$. This immediately yields 
\begin{equation}\nonumber
\nabla_H^\star u(x)  \subseteq \partial_H u (x).
\end{equation}
Moreover $\partial_H u(x)$ is a convex set in $H_1$ for every $x\in \bbG$, then
our claim follows.
\end{proof}
\begin{rem}\rm
The a.e. h-differentiability of an h-convex function $u$ implies that
$\nabla_H^*u(x)\neq\emptyset$ for all $x\in\Omega$. Then \eqref{envstar} 
implies that $\partial_Hu(x)\neq\emptyset$ for all $x\in\Omega$.
This fact was first proved in \cite{CaPi1}.
The opposite implication can be found in \cite{DGN2} for h-convex domains. 
The same implication holds for h-convex functions on open sets,
since the everywhere h-subdifferentiability implies the 
everywhere Euclidean subdifferentiability along horizontal lines.
Then the Euclidean characterization of convexity through the subdifferential
gives the Euclidean convexity  
along horizontal lines, that coincides with the notion of h-convexity.
\end{rem}
%
%
%
%
%
%
%
%
%
%
%
%
%
%
%
%
%
%
%
%
%
%
%
%
%
%
%
%
%
%
%
%
%
%
%
\begin{defin}\label{subjet}\rm
Let $u : \Omega \rightarrow \R$ and $\Omega \subset \bbG$ an open subset. Then we define the first order \textit{sub jet} of $u$ at $x \in \Omega$ as
\begin{equation}\nonumber
J^{1,-}_u (x) = \left\lbrace p \in H_1\ :\ u(xh) \geq u(x) + \left\langle p, h \right\rangle + o(\| h \|),\; \mbox{if}\; [0,h]\subset H_1\cap x^{-1}\Omega \right\rbrace 
\end{equation}
\end{defin}
\begin{rem}\label{eqSJet}\rm
Let $u$ be an h-convex function in $\Omega$. Then $u$ is h-subdifferentiable at $x$ if and only if $J^{1,-}_u(x) \neq \emptyset$. Moreover $J^{1,-}_u(x) = \partial_H u(x)$. For the reader's sake we give the proof of this property, in the Heisenberg group it has been proved in \cite{CaPi1}. The inclusion $J^{1,-}_u(x) \supseteq \partial_H u(x)$ follows by definition. Now let $p \in J^{1,-}_u(x)$, and fix $[0,h] \subseteq x^{-1}\Omega\cap H_1$ . Then $p$ satisfies
$$
u(x\delta_t h) \geq u(x) + \left\langle p, th \right\rangle + o(\| th \|).
$$
By h-convexity of $u$, $tu(xh) + (1-t)u(x) \geq u(x\delta_th)$ which implies
$$
u(xh) \geq u(x) +  \left\langle p, h \right\rangle + \dfrac{o(\|th \|)}{t}.
$$
Now the claim follows letting $t\rightarrow 0$.
\end{rem}
\begin{defin}\label{lambdasub}\rm
Let $u : \Omega \rightarrow \R$ and $\Omega \subset \bbG$ an open subset. Given $\lambda \geq 0$ we define the $\lambda$-subdifferential of $u$ at $x \in \Omega$ as
\begin{equation}\nonumber
\partial^\lambda_H u (x) = \left\lbrace p \in H_1\ :\ u(xh) \geq u(x) + \left\langle p, h \right\rangle - \lambda\| h \|^2,\quad \mbox{whenever } [0,h] \subseteq H_1\cap x^{-1}\Omega \right\rbrace.
\end{equation}
Notice that $\partial^0_H u(x)$ coincides with the h-subdifferential $\partial_H u(x)$.
\end{defin}
\begin{lem}\label{hconvexP}
Consider a function $u = U + P$ in $\Omega$. Let $U$ be h-convex and $P$ a polynomial with $\mbox{h-deg}P \leq 2$, denote by $P^{(2)}$ the $2$-homogeneous part of $P$.  Define $\displaystyle \lambda = \max_{\|w\| = 1} \vert P^{(2)}(w) \vert $, then
\begin{equation}\nonumber
\partial^\lambda_H u(x)  \supseteq \partial_H U(x) + \nabla_HP(x).
\end{equation}
\end{lem}
\begin{proof}
Recall that by Remark \ref{poly}, for every $x,h \in \bbG$ we have
$$
P(xh)  \geq  P(x) + \left\langle \nabla_H P(x), h \right\rangle  - \lambda\| h \|^2.
$$
Let $p \in \partial_H U(x)$ then by definition of h-subdifferential and the previous inequality
$$
U(xh) + P(xh) \geq U(x) + P(x) + \left\langle p + \nabla_H P(x), h \right\rangle  - \lambda\| h \|^2,
$$
whenever $[0,h] \subseteq x^{-1}\Omega \cap H_1$. This implies that 
$p+\nabla_H P(x) \in \partial^\lambda_H u(x)$.
\end{proof}

\begin{prop}\label{hconvexP2}
Let $u = U + V : \Omega \rightarrow \R$, where $U$ is an h-convex function, and $V \in C^1_H(\Omega)$. Fix $\lambda \geq 0$, then for every $x \in \Omega$ we have
\begin{equation}\nonumber
\partial^\lambda_H u(x) \subseteq \partial_H U(x) + \nabla_HV(x).
\end{equation}
\end{prop}
\begin{proof}
In fact let $p \in \partial^\lambda_H u(x)$ and $[0,h] \subseteq H_1\cap x^{-1}\Omega$
\begin{eqnarray}\nonumber
u(xh) &\geq& u(x) + \left\langle p,h\right\rangle - \lambda\|h\|^2 \\ \nonumber
U(xh) + V(xh)  &\geq& U(x) + V(x) + \left\langle \nabla_H V(x), h \right\rangle  + \left\langle p - \nabla_HV(x),h\right\rangle - \lambda\| h \|^2
\end{eqnarray}
Then by the smoothness of $P$ it follows that
\begin{equation}\nonumber
U(xh) \geq U(x) + \left\langle p - \nabla_HV(x),h\right\rangle + o(\| h \|)
\end{equation}
recall that $U$ is h-convex thus by Remark \ref{eqSJet} , $p - \nabla_HV(x) \in \partial_H U(x)$. Therefore the inclusion is proved.
\end{proof}

In the following theorem we extend the classical non-smooth mean value theorem to stratified groups.
\begin{teo}\label{meanval}
Let $u = U + P$, where $U$ is h-convex and $P$ is a polynomial, with $\mbox{h-deg}\,P \leq 2$
and  $\lambda = \max_{\|w\|=1} \vert P^{(2)}(w)\vert$. Then for every $x \in \Omega$ and every $h$ such that $[0,h]\subseteq H_1\cap x^{-1}\Omega$,  there exist $t \in [0,1]$ and $p \in \partial^\lambda_H u (x\delta_th)$ such that
\begin{equation}\nonumber
u(xh) - u(x) =  \left\langle p ,h \right\rangle.
\end{equation}
\end{teo}
\begin{proof}
Let $U_i$ be a sequence of $C^\infty(\Omega)$ h-convex functions, converging to $U$ uniformly on compact sets. Define $u_i = U_i + P$. For such functions the mean value theorem holds i.e. there exists $t_j \in [0,1]$ such that
\begin{equation}\nonumber
u_i(xh)-u_i(x) = \left\langle \nabla_H u_i(x\delta_{t_i}h), h\right\rangle, \quad [0,h] \subset H_1\cap x^{-1}\Omega.
\end{equation}
Possibly passing to a subsequence we have $t_i \rightarrow t$ and $\nabla_H u_i(x\delta_{t_i}h) \rightarrow p$, thus by the uniform convergence
\begin{equation} \nonumber
u(xh)-u(x) = \left\langle  p, h\right\rangle.
\end{equation}
Our claim follows if we prove that $p \in \partial^\lambda_H u(x\delta_th)$. By Proposition \ref{unifConv}, for every $k > 0$ there exists $i_k$ such that 
\begin{equation}\nonumber
\nabla_H U_i(x\delta_{t_i}h) = \partial_H U_i(x\delta_{t_i}h) \subseteq \partial_H U(x\delta_th) + \frac{1}{k} \mathbb B, \qquad \forall i \geq i_k\,
\end{equation}
Moreover, possibly choosing a larger $i_k$, we have
\begin{eqnarray} \nonumber 
\nabla_H U_i(x\delta_{t_i}h) + \nabla_H P(x\delta_{t_i}h) \subseteq \partial_H U(x\delta_{t}h) + \nabla_H P(x\delta_{t}h) +  \frac{2}{k} \mathbb B, \qquad \forall i \geq i_k\
\end{eqnarray}
By Lemma \ref{hconvexP},  $\partial^\lambda_H u (x) \supseteq \partial_H U(x) + \nabla_H P(x)$ thus the previous inclusion implies that
$$
\nabla_H u_i(x\delta_{t_i}h) = \nabla_H U_i(x\delta_{t_i}h) + \nabla_H P(x\delta_{t_i}h)  \subseteq \partial^\lambda_H u(x\delta_th) + \frac{2}{k} \mathbb B, \qquad \forall i \geq i_k\
$$
then letting $k \rightarrow \infty$ we get that $p \in \partial_H^\lambda u(x\delta_{t}h)$.
\end{proof}
\medskip
As an immediate consequence of the previous theorem, we get the following 

\begin{proof}[Proof of Theorem~\ref{meanV}] 
It suffices to apply Theorem~\ref{meanval} with $P=0$ and $\lambda = 0$. 
\end{proof}
\begin{defin}\rm\label{par+}
Let $u : \Omega \rightarrow \R$ and let $h \in H_1$.  The {\em horizontal directional derivative} 
of $u$ at $x$, along $h$, is given by the limit
\[ 
\lim_{\lambda \rightarrow 0^+} \frac{u(x\delta_\lambda h) - u(x)}{\lambda}\,,
\]
whenever it exists. We denote this derivative by $u'(x,h)$.
\end{defin}
\begin{Cor}\label{subadd}
Let $u$ be an h-convex function in $\Omega$. Then for every $x \in \Omega$
and $h\in H_1$ the horizontal directional derivative $u'(x,h)$ exists and satisfies
\begin{equation}\label{dermax}
u'(x,h)= \max_{p \in \partial_H u(x)} \left\langle p, h\right\rangle,
\end{equation}
hence it is subadditive with respect to the variable $h$. 
\end{Cor}
\begin{proof}
The h-convexity of $u$ implies the existence of $u'(x,h)$ for any $x\in \Omega$ and $h\in H_1$. 
Let $p_0 \in \partial_Hu(x)$, such that $\displaystyle \left\langle p_0,h\right\rangle = \max_{p \in \partial_H u(x)} \left\langle p, h\right\rangle$. By definition of $\partial_Hu(x)$, 
$$
u(x\delta_\lambda h) \geq u(x) +  \left\langle p_0, \lambda h\right\rangle, \qquad \mbox{whenever } [0,\lambda h] \subset x^{-1}\Omega\cap H_1.
$$
Then we easily get that
$$
\lim_{\lambda \rightarrow 0^+} \dfrac{u(x\delta_\lambda h) - u(x)}{\lambda} \geq \left\langle p_0, h\right\rangle.
$$
Notice that, for $\lambda$ small enough, $[0,\lambda h] \subset x^{-1}\Omega\cap H_1$, hence we can apply Theorem~\ref{meanval}. Then for every $\lambda$ there exist $c(\lambda) \in [0,1]$ and $p(\lambda) \in \partial_Hu(x\delta_{c(\lambda)\lambda} h)$ such that
$$
\dfrac{u(x\delta_\lambda h) - u(x)}{\lambda} = \left\langle p(\lambda), h\right\rangle.
$$
Now fix a sequence $\lambda_i \rightarrow 0$ such that $p(\lambda_i) \rightarrow\bar p$, then by the closure property of the subdifferential we get $\bar p \in \partial_H u(x)$. Moreover, the existence
of the following limit gives
$$
\lim_{\lambda \rightarrow 0^+} \dfrac{u(x\delta_\lambda h) - u(x)}{\lambda} = 
\left\langle \bar p, h\right\rangle \leq \max_{p \in \partial_H u(x)} \left\langle p, h\right\rangle,
$$
concluding the proof.
\end{proof}
\begin{proof}[Proof of Theorem~\ref{UniqueSub}]
By Proposition \ref{genprop}, there exist $w_s \in \H_1$, $s = 1,\ldots, \gamma$ and and $U\subset \R^\gamma$ open bounded neighbourhood of the origin such that, given $w \in \bbG$, $\|w\| = 1$ then $w = \prod_{s=1}^\gamma a_s w_s$, for an $a \in U$. Fix $r>0$ such that $B_{0,r} \subset x^{-1}\Omega$ and let $M$ be as in Proposition \ref{genprop}. Define the h-convex function 
\begin{equation}\nonumber
g(y) = u(xy)-u(x) - \left\langle p, y \right\rangle, \qquad y \in x^{-1}\Omega.
\end{equation}
Fix $\rho_0 >0$ such that $\rho_0 M < r$. Then for every $\rho < \rho_0$, by 
Theorem~\ref{meanV} and the generating property, we have
\begin{equation}\nonumber
g(\delta_\rho w) = \sum_{s =1}^\gamma \left\langle p_s,\rho a_s w_s\right\rangle - \left\langle p,\rho a_s w_s\right\rangle 
\end{equation}
where $ p_s \in \partial_H u \left(x \delta_\rho(\prod_{k=1}^{s-1} a_k w_k) \delta_{t_s}\delta_\rho a_s w_s\right)$ with $t_s \in [0,1]$. By Proposition \ref{unifConv}, for every $\epsilon >0$ there exists $\rho_0$ such that 
\begin{equation}\nonumber
\partial_H u \left (x\delta_\rho(\prod_{k=1}^{s-1} a_s w_k ) \delta_{t_s}\delta_\rho a_s w_s \right) \subseteq \partial_H u (x) + \epsilon \mathbb B = \{p\} + \epsilon \mathbb B \qquad \forall \rho < \rho_0, \quad s = 1,\ldots,\gamma. 
\end{equation}
Thus $\vert g(\delta_\rho w) \vert  \leq C \gamma \epsilon \rho$ or equivalently $\dfrac{\vert g(\delta_\rho w) \vert}{\rho}$ converges to zero uniformly  in $w$.
\end{proof}
\vskip 0,3 cm
\section{Second order differentiability}
%
%
%
%
%
\begin{rem}\label{Polydiff}\rm
If \eqref{alex} holds for $u$ at $x\in \Omega$, then $u$ is h-differentiable at $x$. In fact we can rewrite \eqref{alex} as $u(xw) - P_x^{(0)}(w)-P_x^{(1)}(w) = P_x^{(2)}(w) + o(\|w\|^2)$.
Clearly $P_x^{(0)}(w) = u(x)$ and $P_x^{(1)}(w)$ is an h-linear map. Thus 
$| u(xw) - u(x)- P_x^{(1)}(w) | = o(\|w\|)$ and the h-differentiability of $u$ follows. 
Moreover by the uniqueness of the h-differential we get that $P_x^{(1)}(w) = \left\langle \nabla_H u(x),w \right\rangle$. 
\end{rem}
As in \cite{Rock00}, we introduce the difference quotients of convex functions.
\begin{defin}[Difference quotients, \cite{Rock00}]\rm
Let $u :\Omega \rightarrow \R$ be h-convex and assume that it is h-differentiable at $x$.
Let $\tau>0$ and define the {\em second h-differential quotient} $\Delta^2_{x,\tau} u$
at $x$ as follows 
\begin{equation}\label{secondDiff}
\Delta^2_{x,\tau}  u(w)=\dfrac{u(x\delta_\tau w) - u(x) - \tau \left\langle \nabla_H u(x), w \right\rangle}{\tau^2}\,.
\end{equation}
Then the h-difference quotient of the subdifferential mapping is given by the
set-valued mapping 
\begin{equation}\label{secondSub}
\Delta_{x,\tau} \partial_H u : w \rightrightarrows \dfrac{\partial_H u(x\delta_\tau w) - \nabla_H u(x)}{\tau}.
\end{equation}
\end{defin}
\begin{rem}\rm
Notice that $\Delta^2_{x,\tau} u$  can be written as
\begin{equation}\nonumber
\Delta^2_{x,\tau} u(w) =   \tau^{-1}\left[ u_{x,\tau}(w) - \left\langle  \nabla_H u(x), w\right\rangle \right] 
\end{equation}
where $u_{x,\tau}(w) = \tau^{-1} \left\lbrace u(x\delta_\tau w) - u(x)\right\rbrace $ and $u_{x,\tau}$ is clearly h-convex. Moreover if we take the subdifferential of $\Delta^2_{x,\tau}u$ we get
\begin{eqnarray}\label{deltasq}
\partial_H \left[ \Delta^2_{x,\tau} u(w)\right] & =& \tau^{-1} \left\lbrace \partial_H u_{x,\tau}(w) - \nabla_H u(x)\right\rbrace \\ \nonumber
&  = & \tau^{-1} \left\lbrace \partial_H u(x\delta_\tau w) - \nabla_H u(x)\right\rbrace \\
& = & \Delta_{x,\tau} \partial_H u (w) \nonumber .
\end{eqnarray}
where the equality $\partial_H u_{x,\tau}(w) = \partial_H u(x\delta_\tau w)$ follows from the definition of $u_{x,\tau}$.
\end{rem}
\medskip
\begin{proof}[Proof of Lemma~\ref{MIeqGR}]
Choosing $w=0$ we get $\partial_Hu(x) = \{v\}$, thus by Theorem~\ref{UniqueSub}, $u$ is h-differentiable at $x$, moreover $v = \nabla_H u(x)$. The twice h-differentiability immediately follows from \eqref{mignot}, taking its restriction to all h-differentiability points. 
For the converse implication, we rewrite expansion \eqref{difgrad} as follows, for all $\epsilon >0$ there exists $\rho >0$ such that 
\begin{equation}\label{inc1}
\left\vert \dfrac{\nabla_H u(xh) - \nabla_H u(x) - A_x(h)}{\|h\|}  \right\vert  \leq \epsilon \qquad \|h\| < \rho.
\end{equation}
for all $h \in x^{-1}\Omega$ such that $u$ is h-differentiable at $xh$.
By \eqref{limgrad}, for any $w \in x^{-1}\Omega \cap B_{0,\rho}$, taking into account \eqref{inc1}, we get 
\begin{equation}\nonumber
\left \vert  \dfrac{p - \nabla_H u(x) - A_x(w)}{\|w\|} \right \vert  \leq \epsilon  \qquad \mbox{for all } p \in \nabla_H^\star u(xw).
\end{equation}
In an equivalent form, we have
\begin{equation}\label{inclstar}
\nabla_H^\star u(xw) \subseteq \nabla_H u(x) + A_x(w) + \epsilon \|w\|\mathbb B.
\end{equation}
Moreover, the set on the right is convex thus, Theorem \ref{carsubdif} yields
\begin{equation}
\partial_H u(xw) = \bar{co}\left( \nabla_H^\star u(xw)\right)  \subseteq \nabla_H u(x) + A_x(w) + o(\|w\|)\mathbb B.
\end{equation}
This leads us to the conclusion.
\end{proof}
\begin{Cor}\label{eqMi}
$u$ is twice h-differentiable at $x$ if and only if, for any bounded set $W \Subset \Omega$, for all $\epsilon > 0$ there exists $\delta > 0$  such that
for all $w \in W$ and $\tau \in (0,\delta)$ we have 
\begin{equation}\label{equationMi}
\emptyset \neq \Delta_{x,\tau} \partial_H u (w) - A_x(w) \subseteq \epsilon \mathbb B.
\end{equation}
\end{Cor}
\begin{proof}
Let $u$ twice h-differentiable at $x$, fix a bounded set $W \Subset \Omega$ and $\epsilon >0$. Then there is $\rho(\epsilon) > 0$ such that 
\begin{equation}\nonumber
\partial_Hu(xw) \subset \nabla_H u(x) + A_x(w) +   \|w\|\epsilon \mathbb B, \qquad \|w\| < \rho(\epsilon).
\end{equation}
If $w= \delta_\tau h$, with $h \in W$, then for $\tau < \frac{\rho(\epsilon)}{\mbox{diam}(W)}$
$$
\partial_Hu(x\delta_\tau h) \subset \nabla_Hu(x) + \tau A_x(h) + \epsilon \tau \mbox{diam}(W)\mathbb B
$$
which is equivalent to \eqref{equationMi}.
Conversely, suppose that \eqref{equationMi} holds for $W = \{w \in \G,\ \|w\| = 1\}$ and $\epsilon>0$ fixed. Then there exists $\delta >0$ such that for every $\tau>0$
$$
\dfrac{\partial_Hu(x\delta_\tau w) - \nabla_H u(x)}{\tau} - A_x(w) \subseteq \epsilon \mathbb B.
$$
Notice that the previous inclusion holds for every $\|h\| \leq \delta$, i.e.
$$
\partial_Hu(x h) \subseteq \nabla_H u(x) + A_x(h) + \epsilon \|h\|\mathbb B,
$$
this concludes the proof.

\end{proof}

\vskip 0,4 cm

\begin{proof}[Proof of Theorem \ref{equivalence}]
Define $\phi(w) := P_x^{(2)}(w)$ to be the 2-homogeneous part of $P_x$, notice that $\nabla_H P_x^{(2)}(w)$ is an h-linear map, since it is a polynomial of homogeneous degree 1. Let us show that $U_{x,\tau} := \Delta^2_{x,\tau} u$ uniformly converges on compact sets to $\phi$. We, fix a compact set $K \subseteq \Omega_x$, and consider the difference $U_{x,\tau}- \phi(w)$. By Remark \ref{Polydiff}, we get
\begin{equation}\label{2exp}
U_{x,\tau}- \phi(w)= \dfrac{u(x\delta_\tau w) - P_x^{(0)}(\delta_\tau w)- P_x^{(1)}(\delta_\tau w) - P_x^{(2)}(\delta_\tau w)}{\tau^2} = \dfrac{o(\|\delta_\tau w \|^2)}{\tau^2}.
\end{equation}
Moreover $U_{x,\tau}$ is h-convex, then so is $\phi$. By Proposition \ref{unifConv}, for every compact set $W \subset \Omega$ and $\epsilon >0$ there exists $\gamma > 0$ such that
\begin{equation}\nonumber
\partial_H U_{x,\tau}(w) \subseteq \nabla_H \phi(w) + \epsilon \mathbb B, \quad
\mbox{for all}\quad w \in W\quad\mbox{and}\quad \tau \in (0,\gamma).
\end{equation}
Notice that \eqref{deltasq}, gives
\[
\partial_H U_{x,\tau} (w) = \Delta_{x,\tau} \partial_H u (w).
\]
Thus, taking into account that $\phi = P_x^{(2)}$. It follows that 
\[
\Delta_{x,\tau}\left[ \partial_H u\right](w) \subseteq \nabla_H{P^{(2)}_x}(w) + \epsilon \mathbb B,
\]
hence
$\Delta_{x,\tau}\left[ \partial_H u \right](w) - \nabla_H{P^{(2)}_x}(w)  \subset \epsilon \mathbb B.$
By Corollary \ref{eqMi}, $u$ is twice h-differentiable.

Now, we assume that $u$ is twice h-differentiable at $x$. 
Then Lemma~\ref{MIeqGR} give us an h-linear $A_x$ such that
$$
\nabla_H u(xw) = \nabla_H u(x) + A_x(w) + o(\|w\|).
$$
Recall that by Proposition \ref{genprop} we can find an integer $\gamma$, and an open bounded neighbourhood of the origin $U \subset \R^\gamma$ such that 
\begin{equation}\nonumber
W = \left\lbrace \prod_{i=1}^\gamma  a_s w_{i_s},\quad a \in U,\ w_{i_s} \in  H_1\right\rbrace \supset  B_{0,1}.
\end{equation}
Define $v$ as $v(w) = u(xw)-u(x)- P_x(w)$
where $P_x(w)$ is the unique polynomial, with h-$\mbox{deg}P \leq 2$, such that 
\begin{equation}\label{nablaP}
\nabla_H P_x(w) = \nabla_H u(x) + A_x w
\end{equation}
and $P_x(0) = 0$, as a consequence of Remark \ref{uniqueP}.
Let $r>0$, such that $B_{0,r} \subset x^{-1}\Omega$ and define $M$ as in Proposition \ref{genprop}. Let $\rho_0$ such that $\rho_0 M < r$ and consider $w$, $\|w\| = 1$. Then for every $\rho < \rho_0$,  $v(\delta_\rho w) = v(\delta_\rho w)-v(0)$ can be written as
\begin{equation}\nonumber
v(\delta_\rho  w) = \sum_{s=1}^\gamma v(\prod_{l=1}^s \delta_\rho a_{i_l} w_{i_l}) - v(\prod_{l=1}^{s-1}\delta_\rho a_{i_l} w_{i_l}).
\end{equation}
Observe that $v$ is an h-convex function plus a polynomial of homogeneous degree less than or equal to 2, thus by Theorem \ref{meanval} applied to horizontal directions $w_s$ we get
\begin{equation}\nonumber
v(\delta_\rho w) = \sum_{i=1}^\gamma \left\langle p_s , \delta_\rho a_s w_s \right\rangle 
\end{equation}
with $p_s \in \partial^{\lambda}_H v\left(x\delta_\rho (\prod_{i=1}^{s-1} a_i w_i) \delta_{t_s} \delta_\rho a_s w_s\right)$, $\displaystyle \lambda = \max_{\| h \| =1} \vert P_x^{(2)}(h) \vert $, for suitable $t_s \in [0,1]$. Moreover, by Proposition \ref{hconvexP2} we know that 
\begin{equation}\label{th2e1}
 p_s  + \nabla P_x \left(\delta_\rho(\prod_{i=1}^{s-1} a_i w_i )\delta_{t_s} \delta_\rho a_s w_s\right) \in \partial_H u\left(x\delta_\rho(\prod_{i=1}^{s-1}a_i w_i) \delta_{t_s}\delta_\rho a_s w_s\right) .
\end{equation}
The expansion \eqref{mignot} for the h-subdifferential of $u$ implies that
\begin{eqnarray}\label{th2e2}
\partial_H u \left(x \delta_\rho(\prod_{i=1}^{s-1}a_i w_i) \delta_{t_s}\delta_\rho a_s w_s\right) &\subset& \nabla_H u(x) + A_x \left(\delta_\rho(\prod_{i=1}^{s-1}a_i w_i) \delta_{t_s}\delta_\rho a_s w_s\right)\\ \nonumber
 & &+ o\left( \vert \delta_\rho( \prod_{i=1}^{s-1} a_i w_i) \delta_{t_s} \delta_\rho a_s w_s \vert\right) \mathbb B,
\end{eqnarray}
thus by the choice of $P_x$ and taking into account \eqref{th2e1} and \eqref{th2e2}, we get that 
$$\vert p_s \vert  = o\left( \vert \delta_\rho(\prod_{i=1}^{s-1}a_i w_i)  \delta_{t_s} \delta_\rho  a_s w_s \vert\right) = o (\rho). 
$$ 
As a consequence, $\vert v(\delta_\rho w) \vert  = o(\rho^2)$ and our equivalence is achieved. 

Finally, we have to prove claims (1), (2) and (3). 
The first one follows considering the restriction of \eqref{2exp} to directions $w\in V_2$
and taking into account \eqref{exppol}, hence getting
the uniform limit
\[
\frac{u(x\cdot \exp(t^2 W))-u(x)-t^2\lan\nabla_{V_2}P_x^{(2)},w\ran}{t^2} \lra 0
\]
as $t\to0^+$, where $w$ varies in a compact neighbourhood of zero in $V_2$.
In fact, we have used the equality 
\[
x\delta_tw=x\cdot\delta_t\exp(W)=x\cdot\exp(t^2W),
\]
where $W$ is the unique left invariant vector field such that $W(0)=w$.
In particular, we have $\nabla_{V_2}u(x)=\nabla_{V_2}P$.
Taking into account Remark~\ref{Polydiff} and formula \eqref{exppol}, then claim (2) follows.
Now, with respect to the fixed basis $(e_1,\ldots,e_n)$ of $\G$, we have coefficients $(A_x)_j^i$
such that
\[
A_xw=\sum_{i,j=1}^{m_1} (A_x)_j^i\,w_i\,e_j\,,
\]
therefore \eqref{nablaP} yields $\nabla_HP_x^{(2)}(w)=A_xw$. Precisely, for any $j=1,\ldots,m_1$, we have
\[
X_jP^{(2)}_x(w)=\sum_{i=1}^{m_1}(A_x)^i_j\,w_i\,,
\]
then formula \eqref{alij} gives
\[
X_iX_jP^{(2)}_x=(A_x)^i_j=(\nabla_H^2P^{(2)}_x\big)_{ij}+\sum_{l=m_1+1}^{m_2} X_lu(x)\;a^{li}_j\,.
\]
As a result, we get
\[
(\nabla_H^2P^{(2)}_x)_{ij}=(A_x)^i_j-\sum_{l=m_1+1}^{m_2} X_lu(x)\;a^{li}_j\,,
\]
that coincides with the formula of claim (3). Finally, we observe that $P_x^{(2)}$ is the uniform
limit on compact sets of the h-convex functions $U_{x,\tau}$. This implies that
$P_x^{(2)}$ is also h-convex and then its symmetrized horizontal Hessian is nonnegative.
\end{proof}


\begin{thebibliography}{99}
\bibitem{AlbAmb99}{\sc  G.Alberti, L.Ambrosio},
{\em A geometrical approach to monotone functions in $\R^n$},
Math. Z., {\bf 230}, 259-316, (1999)

\bibitem{Alek39}{\sc A.D.Aleksandrov},
{\em Almost everywhere existence of the second differential
of a convex function and some properties of convex surfaces
connected with it},
Leningrad Univ. Ann. (Math. ser.), {\bf 6},3-35, (1939) (in Russian)

\bibitem{AmbDan00} {\sc L. Ambrosio, N. Dancer}, {\em Calculus of variations and partial differential equations. Topics on geometrical evolution problems and degree theory}. Papers from the Summer School held in Pisa, September 1996. Edited by G. Buttazzo, A. Marino and M. K. V. Murthy; Springer-Verlag, Berlin, (2000).

\bibitem{AmbMag}{\sc L.Ambrosio, V.Magnani},
{\em Weak differentiability of BV functions on stratified groups},
Math. Z., {\bf 245}, 123-153, (2003)

\bibitem{BalRic03}{\sc Z.M.Balogh, M.Rickly},
{\em Regularity of convex functions on Heisenberg groups},
Ann. Sc. Norm. Super. Pisa Cl. Sci. (5) {\bf 2}, n.4, 847-868, (2003)

\bibitem{BCP96}{\sc G. Bianchi, A. Colesanti, C. Pucci},
{\em On the second order differentiability of convex surfaces},
Geom. Dedicata, {\bf 60}, 39-48, (1996) 
\bibitem{BusFel35}{\sc H. Busemann, W. Feller},
{\em Kr\"ummungseigenschaften konvexer Fl\"achen},
Acta Math., {\bf 66}, 1-47, (1935).

\bibitem{CafCab95}{\sc L.A.Caffarelli, X.Cabr\'e}, 
{\em Fully nonlinear elliptic equations}, 
AMS Colloquium Publications, {\bf 43},
AMS, Providence, RI, (1995)

\bibitem{CaPi1}{\sc A. Calogero, R. Pini}, 
{\em Horizontal Normal Map on the Heisenberg group}, arXiv:0811.2277.

\bibitem{Clarke90}{\sc F. H. Clarke},
{\em Optimization and nonsmooth analysis}, SIAM, (1990).

\bibitem{DGN03}{\sc D.Danielli, N.Garofalo, D.M. Nhieu},
{\em On the best possible character of the $L^Q$ norm in some a priori estimates for non-divergence form equations in Carnot groups}, Proc. Amer. Math. Soc. {\bf 131}, n.11, 3487-3498, (2003)

\bibitem{DGN2}{\sc D.Danielli, N.Garofalo, D.M. Nhieu},
{\em Notions of convexity in Carnot groups};  Comm. Anal. Geom. {\bf 11}, n.2,
263-341, (2003).

\bibitem{DGNT}{\sc D.Danielli, N.Garofalo, D.M. Nhieu, F.Tournier}
{\em The theorem of Busemann-Feller-Alexandrov in Carnot groups},
Comm. Anal. Geom. {\bf 12}, n.4, 853-886, (2004)

\bibitem{FS82}{\sc G.B.Folland, E.M. Stein},
{\em Hardy Spaces on Homogeneous groups}; Princeton University Press, (1982).

\bibitem{GarTou05}{\sc N.Garofalo, F.Tournier},
{\em New properties of convex functions in the Heisenberg group}, 
Trans. Am. Math. Soc., {\bf 358}, n.5, 2011-2055, (2005)

\bibitem{GutMon1} {\sc C.E. Guti\'errez, A. Montanari}, 
{\em Maximum and comparison principle for convex functions on the Heisenberg group},
Comm. Partial Differential Equations, {\bf 29} ,no. 9-10, 1305-1334, (2004).

\bibitem{GutMon2} {\sc C.E. Guti\'errez, A. Montanari}, {\em On the second order derivatives of convex functions on the Heisenberg group}; Ann. Sc. Norm. Super. Pisa Cl. Sci. 5, 2, 349-366, (2004).

\bibitem{JLMS} {\sc P. Juutinen, G. Lu, J.J. Manfredi, B. Stroffolini}, 
{\em Convex functions on Carnot groups}, Rev. Mat. Iberoam. {\bf 23},  no. 1, 191-200, (2007).

\bibitem{LMS} {\sc G. Lu, J.J. Manfredi, B. Stroffolini}, {\em Convex functions on the Heisenberg group}; Calc. Var. 19, 1-22, (2004).

\bibitem{Mag09} {\sc V. Magnani}, 
{\em Contact equations, Lipschitz extensions and isoperimetric inequalities },
Calc. Var. Partial Differential Equations, {\bf 39}, 233–271, (2010)

\bibitem{Mag} {\sc V. Magnani}, {\em Lipschitz continuity, Aleksandrov theorem and characterizations for $H$-convex functions}, Math. Ann. {\bf 334}, 199--233, (2006). 

\bibitem{MagTh}{\sc V. Magnani}, {\em Elements of Geometric Measure Theory on Sub-Riemannian Groups}; Scuola Normale Superiore Pisa, (2002).

\bibitem{Mig76}{\sc F. Mignot}, 
{\em Contr\^ole optimal dans les in\'equations variationelles elliptiques}, J.
Funct. Anal. {\bf 22}, 130–185, (1976)


\bibitem{Pan89}{\sc P.Pansu},
{\em M\'etriques de Carnot-Carath\'eodory quasiisom\'etries des
espaces sy\-m\'e\-tri\-ques de rang un}, Ann. Math., {\bf 129},
1-60, (1989)

\bibitem{Resh68}{\sc Yu. G. Reshetnyak},
{\em Generalized derivatives and differentiability almost everywhere},
Math. USSR-Sb. {\bf 4}, 293-302 (1968)

\bibitem{Rickly06}{\sc M. Rickly}, {\em First-order regularity of convex functions on
Carnot groups}, J. Geom. Anal. {\bf 16}, n.4, 679-702 (2006). 

\bibitem{Rock85}{\sc R. T. Rockafellar}, 
{\em Maximal monotone relations and the second derivatives
of convex functions}, Ann. Inst. H. Poincar´e, Analyse non lin\'eaire,
{\bf 2}, 167-184, (1985)

\bibitem{Rock97}  {\sc R. T. Rockafellar, R. J. Wets}, {Variational Analysis}; Springer, (1997).

\bibitem{Rock00}{\sc R. T. Rockafellar}, 
{\em Second Order Convex Analysis}, NonLinear and Convex An. {\bf 1}, 1-16, (2000).

\bibitem{Trud05}{\sc N.S. Trudinger}, {\em On Hessian measures for non-commuting vector fields},
Pure Appl. Math. Q. {\bf 2}, n. 1, part 1, 147-161, (2006) 

\bibitem{Wang05}{\sc C. Wang},
{\em Viscosity convex functions on Carnot groups},
Proc. Amer. Math. Soc., {\bf 133}, 1247–1253 (2005) 

\end{thebibliography}
\end{document}